\documentclass[dvipdfmx, 12pt, a4paper]{amsart}

\usepackage[utf8]{inputenc}
\usepackage[T1]{fontenc}
\usepackage{amsmath,amssymb,amsthm}
\usepackage{amsfonts}
\usepackage{graphicx}

\usepackage{tikz}
\usetikzlibrary{cd}

\numberwithin{equation}{section}

\theoremstyle{plain}
\newtheorem{thm}{Theorem}[section]
\newtheorem{prop}[thm]{Proposition}
\newtheorem{lem}[thm]{Lemma}

\newtheorem{question}[thm]{Question}

\theoremstyle{definition}
\newtheorem{defi}[thm]{Definition}
\newtheorem*{defi*}{Definition}

\newtheorem*{ack}{Acknowledgments}

\theoremstyle{remark}
\newtheorem{rmk}[thm]{Remark}

\newcommand{\ra}{\rightarrow}

\newcommand{\cO}{\mathcal{O}}
\newcommand{\cP}{\mathcal{P}}

\newcommand{\fD}{\mathfrak{D}}

\newcommand{\bC}{\mathbb{C}}

\newcommand{\bP}{\mathbb{P}}
\newcommand{\bQ}{\mathbb{Q}}

\newcommand{\bZ}{\mathbb{Z}}

\DeclareMathOperator{\Pic}{Pic}

\DeclareMathOperator{\Supp}{Supp}

\DeclareMathOperator{\Spec}{Spec}

\DeclareMathOperator{\mult}{mult}

\DeclareMathOperator{\diag}{diag}
\DeclareMathOperator{\chara}{char}
\DeclareMathOperator{\rank}{rank}

\newcommand{\val}{\mathit{val}}

\begin{document}
	
\title[On the Standard Models of Degree 4 del Pezzo Fibrations]{On the Standard Models of del Pezzo Fibrations of Degree four}
\author{Natsume Kitagawa}
\address{Graduate School of Mathematics, Nagoya University, Furocho Chikusa-ku, Nagoya, 464-8602, Japan}
\email{natsume.kitagawa.e6@math.nagoya-u.ac.jp}
\maketitle

  \begin{abstract}
  	Corti defined the notion of standard models of del Pezzo fibrations, and studied their existence over $\bC$ with a fixed generic fibre in \cite{Cor}. In this paper, we prove the existence of standard models of del Pezzo fibrations of degree $4$ in characteristic $>2$. To show this, we use the notion of Koll\'ar stability, which was introduced in \cite{Kol1} and \cite{AFK}.
  \end{abstract}

  \tableofcontents

  \section{Introduction}
  
    In birational geometry, \emph{Mori fibre spaces} (MFS) are essential objects. They appear naturally as outputs of the minimal model program (MMP). Finding their good birational models is an interesting problem and is useful for birational geometry.
    In dimension three, there are three cases of MFS: Fano threefolds with Picard number $1$, conic bundles, and del Pezzo fibrations. In this paper, we study the \emph{standard models of del Pezzo fibrations}, which were defined by Corti in \cite{Cor}.

    \subsection{Standard models of del Pezzo fibrations}
    
      In this section, we work over an algebraically closed field $k$ of arbitrary characteristic. Let $C$ be an essentially smooth irreducible one-dimensional scheme over $k$ and $K$ be the function field of $C$. We say that a flat projective morphism $\pi:X\ra C$ is a \emph{del Pezzo fibration} of degree $d$ if its generic fibre $X_K$ is a smooth del Pezzo surface over $K$ of degree $d:=K_{X_K}^2$.
      
      Corti defined the \emph{standard models of del Pezzo fibrations} and discussed their the existence in \cite{Cor}.
      
      \begin{defi}
      	Suppose that $\pi:X\ra C$ is a del Pezzo fibration of degree $d$ with the generic fibre $X_K$. We say that $\pi:X\ra C$ is a \emph{standard model} of $X_K$ over $C$, or a \emph{standard del Pezzo fibration of degree} $d$, if the following conditions are satisfied:
      	\begin{enumerate}
      		\item $X$ has only terminal singularities,
      		\item $\pi$ has integral fibres, and
      		\item $-aK_X$ is a $\pi$-ample line bundle, where
      		$$a=
      		\begin{cases}
      			1\quad(d\geq 3),\\
      			2\quad(d=2),\\
      			6\quad(d=1).
      		\end{cases}$$
      	\end{enumerate}
      \end{defi}
  
      \begin{thm}\cite{Cor}\label{CortiThm}
      	Assume $k=\bC$. Let $C$ be a smooth curve over $k$ and $K$ be the function field of $C$. Suppose $X_K$ be a smooth del Pezzo surface of degree $d\geq 2$ over $K$. Then there exists a standard model $\pi:X\ra C$ of del Pezzo fibrations with the generic fibre $X_K$.
      \end{thm} 
      
      Corti predicted that, when the generic fibre is a del Pezzo surface of degree $1$, there also exist standard models of del Pezzo fibrations. This conjecture was solved in \cite{AFK}.

      \begin{thm}\cite{AFK}\label{AFKdegree1}
      	Assume that $k$ is an algebraic closed field with $\chara k\neq 2, 3$. Let $C$ be a smooth curve over $k$ and $K$ be the function field of $C$. Suppose $X_K$ is a smooth del Pezzo surface of degree $1$ over $K$. Then there exists a standard model $\pi:X\ra C$ of del Pezzo fibrations with generic fibre $X_K$.
      \end{thm}

      MMP for threefolds runs also in characteristic $p>5$ by \cite{Bi}, \cite{BW}, \cite{CTX} and \cite{HX}. Thus we are interested in the following question:

      \begin{question}
      	In positive characteristics, does there exist a standard model of del Pezzo fibrations as a birational model for a given del Pezzo fibration $\pi:X\ra C$ ?
      \end{question}

      Regarding this problem, we deal the case of degree $4$ del Pezzo fibrations. We obtain the following result.

      \begin{thm}\label{MainThm}
      	Suppose that $k$ is an algebraic closed field with $\chara k\neq 2$. Let $C$ be a smooth curve over $k$ and $K$ be the function field of $C$. Suppose $X_K$ be a smooth del Pezzo surface of degree $4$ over $K$. Then there exists a standard model $\pi:X\ra C$ of del Pezzo fibrations with the generic fibre $X_K$.
      \end{thm}

      To prove theorem \ref{CortiThm}, Corti used MMP of threefolds and Kawamata-Vieweg vanishing, which are not available in positive or low characteristics. Hence it seems to be difficult to prove Theorem \ref{MainThm} in a similar way as Corti did, particularly in low characteristics. Therefore we use the notion of \emph{Koll\'ar stability}, which was introduced in \cite{AFK}. Theorem \ref{AFKdegree1} was proved as an application of Koll\'ar stability.

     \begin{rmk}\label{ReductionOverDVR}
     	Using techniques of descent, in order to prove the existence of standard models, we only have to show the existence of standard models $\pi:X\ra C=\Spec R$ with a fixed generic fibre $X_K$, where $R$ is the local ring of a closed point on a curve. See \cite[Remark 1.3]{AFK} and \cite{Vi}.
     \end{rmk}

    \subsection{Semistability over curves}
    
      Koll\'ar defined the notion of semistability for hypersurfaces over $1$-dimensional regular schemes in \cite{Kol1}. This is analogous to GIT-stability for hypersurfaces. As an application, He explained the existence of standard models of degree $3$ del Pezzo fibrations, from the view of stability for hypersurfaces over curves.
      
      \begin{thm}\cite{Kol1}
      	Suppose that $k$ is an algebraic closed field of arbitrary characteristic.
      	\begin{enumerate}
      		\item Let $C$ be a smooth curve over $k$ and $K$ be the function field of $C$. Suppose that $X_K$ is a smooth del Pezzo surface of degree $3$ over $K$. Then there exists a semistable del Pezzo fibration $\pi:X\ra C$ with generic fibre $X_K$.
      		\item Any semistable del Pezzo fibration of degree $3$ is a standard model. 
      	\end{enumerate}
      \end{thm}

      Koll\'ar stability, which was introduced in \cite{AFK}, is a generalization of stability for hypersurfaces over regular $1$-dimensional bases. The notion of Koll\'ar stability enables us to obtain a \emph{good} birational model of a given fibration $\pi:X\ra\Spec R$.
      
      Suppose that $R$ is a DVR and $K$ is the fraction field of $R$. Let $\pi:X\ra\Spec R$ be a given fibration and $M$ be a proper parameter space over $R$ such that the generic fibre $X_K$ of $\pi$ is a $K$-valued point of $M$. Since $M$ is proper, the fibration $X\ra\Spec R$ appears as a pullback of the universal family over $M$ via a morphism $f_\pi:\Spec R\ra M$. Assume that a group scheme $G$ over $R$ acts on $M$. Choose a $G$-linearized line bundle $L\in\Pic^G(M)$ and $G$-invariant nonzero section $\fD\in H^0(M, L)^G$. Let $t$ be a uniformizer of $R$.
      
      \begin{defi}\cite[Definition 1.6]{AFK}
      	A fibration $\pi:X\ra\Spec R$ is called a $\fD$-\emph{semistable} model of $X_K$ if ;
      	\begin{enumerate}
      		\item $f_\pi(Spec K)\notin\Supp(\fD)$, and
      		\item The $t$-valuation of the Cartier divisor $f_\pi^*(\fD)$ on $\Spec R$ is minimal along all $R$-valued points $f:\Spec R\ra M$ with $f_\pi(\Spec K)\in G(K)\cdot f(\Spec K)$.
      	\end{enumerate}
      \end{defi}
  
      Typically $M$ is a parameter space of (systems of) polynomials over $R$, $G$ is the general linear group $GL_n$ over $R$, and the action of $G$ on $M$ is given by the coordinate change. If we choose $M=\bP H^0(\bP^3_R, \cO(3))$, $G=GL_4$ and $\fD$ to be the discriminant divisor, then this setting reconstructs the notion of stability for hypersurfaces over curves by Koll\'ar.

      \vspace{3mm}

      It is known that any smooth degree four del Pezzo surface can be written as a complete intersection of two quadrics ($(2, 2)$-complete intersection) in $\bP^4$. In order to construct the standard models of degree $4$ del Pezzo fibrations, we consider Koll\'ar stability on the moduli space $M$ of $(2, 2)$-complete intersections. We have the following result.

      \begin{thm}\label{MainThm2}
      	Suppose that $k$ is an algebraic closed field with $\chara k\neq 2$. Let $C$ be a smooth curve over $k$ and $K$ be the function field of $C$. Suppose $X_K$ be a smooth $(2, 2)$-complete intersection over $K$ with $\dim X_K\geq2$. Then there exists a model $\pi:X\ra C$ of a fibration with the generic fibre $X_K$ which satisfies the following properties:
      	\begin{enumerate}
      		\item $\pi$ has integral fibres.
      		\item $-K_X$ is a $\pi$-ample line bundle.
      		\item $X$ is regular in codimension $2$, and every singular point of $X$ is a hypersurface singularity.
      	\end{enumerate}
      \end{thm}

      If the generic fibre $X_K$ in Theorem \ref{MainThm2} is a del Pezzo surface of degree $4$, we can investigate singularities on a semistable model $X$. Thus we can prove Theorem \ref{MainThm}.

    \subsection{Related works}
    
    There are numerous works studying del Pezzo fibrations. The moduli spaces of degree $4$ del Pezzo fibrations were studied in \cite{HKT}. Kresch and Tschinkel explicitly studied del Pezzo fibrations of degree $6$, $8$, and $9$ in \cite{KT22}, \cite{KT19}, and \cite{KT20} respectively.
    
    The moduli spaces of rational curves on del Pezzo fibrations with terminal singularities were studied in \cite{LT24} and \cite{LT22}, over fields of characteristic zero. Since the investigation of moduli of rational curves is also considered in positive characteristics as in \cite{BLRT}, our results may provide some examples which are useful to study the moduli spaces of rational curves on del Pezzo fibrations with terminal singularities in positive characteristics.

  \begin{ack}
  	This paper is based on the author's master thesis \cite{Kit}.
  	
  	The author is deeply grateful to his advisor, Professor Sho Tanimoto, for helpful discussions and continuous support. The author is also grateful to Hamid Abban and Yuri Tschinkel for their careful reviews and constructive feedback.
  	
  	The author would like to thank Kaito Kimura, Shu Nimura, Yuya Otake, Arashi Sakai, and Naoki Wakasugi for helpful discussions on commutative algebra. The author is thankful to Runxuan Gao for his valuable assistance in refining English language of this manuscript.
  	
  	
  	The author was partially supported by JST FOREST program Grant number JPMJFR212Z.
  \end{ack}

  \section{Preliminaries}
  
    \subsection{The moduli space of pencils}
    
    Suppose that $k$ is an algebraically closed field with $\chara(k)\neq 2$. Let $R$ be a local ring of a closed point on a smooth curve over $k$. Clearly, $R$ is a DVR. Let $V$ be the space of (single) quadrics in $n$ variables $x_1,\ldots, x_n$ over $R$. Then $V$ is a free $R$-module of rank $r:=n(n+1)/2$. In order to find good models of $(2, 2)$-complete intersections over $R$, we use the moduli of pencils $M:=Gr(2, V)$ defined over $\Spec R$. In this paper, a pencil means an $R$-point of $M$.
    
    Take an $R$-basis $\{e_1,\ldots, e_r\}$ of $V$. Then we can take an $R$-basis $\{e_i\wedge e_j\}$ $(1\leq i<j\leq r)$ of $\wedge^2 V$. Using this $R$- basis of $\wedge^2 V$ and Pl\"ucker embedding $\iota:M\ra\bP(\wedge^2 V)\cong\bP^{s}_R$ (where $s=r(r-1)/2-1$), we can give coordinates to $R$-points of $M$.
    
    Let $G$ be the group scheme $GL_n$ over $\Spec R$. Then $G$ acts naturally on $M$ by coordinate changes.

    \subsection{The discriminant of pencils}

    Let $\cP\in M(R)$ be a pencil. Take any $R$-basis $f, g\in R[x_1,\ldots, x_n]$ of $\cP$ as a pencil and let $A, B\in M_n(R)$ be symmetric matrices corresponding to $f$ and $g$ respectively. The \emph{discriminant of pencil} $\cP$ is a polynomial in $\lambda$ and $\mu$ defined as follows:
    $$\Delta(\lambda, \mu):=\det(\lambda A+\mu B).$$
    Let $X_K\subset\bP_K^{n-1}$ be a $(2, 2)$-complete intersection defined by $f=g=0$ over $K$. By \cite[Proposition 2.1]{Re1}, $X_K$ is smooth if and only if $\Delta(\lambda, \mu)=\det(\lambda A+\mu B)$ is not identically zero and has $n$ distinct roots. Note that the discriminant $\Delta(\lambda, \mu)$ of $\cP$ depends on the choice of the basis $f, g$.

  \section{Stability on the moduli space of pencils over DVR}

    We use the terminology in \cite{AFK} as follows.
    By the valuative criterion of properness, canonically we may identify $R$-valued points of $M$ with $K$-valued points of $M$ i.e., we have a canonical bijection between the sets $M(R)$ and $M(K)$. We denote by $\cP(\Spec K)$ the $K$-point of $M$ corresponding to an $R$-point $\cP$ of $M$. Using this identification, an $G(K)$-action on $M(R)$ is defined as follows: for $\cP\in M(R)$ and $\rho\in G(K)$, define $\rho\cdot\cP\in M(R)$ to be the $R$-point corresponding to the $K$-point $\rho\cdot\cP(\Spec K)\in M(K)$. The $R$-point $\rho\cdot\cP$ is called a \emph{model of} $\cP$. The set of all models of $\cP$ is the $G(K)$-orbit of $\cP$ in $M(R)$.

  \subsection{Multiplicity and semistability}

    Let $t\in R$ be a uniformaizer of $R$.
    
    \begin{defi}
    	We use the notations defined in section 2.1 and the above. 
    	Let $\cP$ be an $R$-point of $M$ and take a $K$-point $\rho$ of $G$. Then we obtain an $R$-point $\rho\cdot\cP$ and $K$-point $\rho\cdot\cP(\Spec K)$ of $M$. By Pl\"ucker embedding, the $K$-point $\rho\cdot\cP(\Spec K)$ is written as
    	$$[a_0:a_1:\cdots:a_s]\qquad(a_i\in K).$$
    	Now, using the uniformizer $t\in R$, we define \emph{multiplicity} of $\cP\in M(R)$ with respect to $\rho\in G(K)$ as follows:
    	$$\mult_\rho(\cP):=\max\{N\,|\, \text{Every}\hphantom{0} t^{-N}a_i\hphantom{0}\text{is an element of}\hphantom{0}R.\}$$
    	(Note that, by Pl\"ucker embedding, the $R$-point $\rho\cdot\cP$ is written as $[t^{-m}a_0:t^{-m}a_1:\cdots:t^{-m}a_s]$ where $m=\mult_\rho(\cP)$.)
    \end{defi}

    The \emph{weight system} $\rho=(w_1,\ldots, w_n)$ is an element of $\bZ^n$. Each weight system determines a $K$-point $\diag(t^{w_1},\ldots, t^{w_n})$ of $G$.
    By the elementary divisor theorem, instead of considering all actions of $K$-point of $G$, we may deal all $K$-points of $G$ defined by some weight systems and all coordinates. For a weight system $\rho=(w_1,\ldots, m_n)$ and a pencil $\cP\in M(R)$, we write simply $\rho\cdot\cP:=\diag(w_1,\ldots, m_n)\cdot\cP$.
    
    Now we define semistability for $R$-points of $M$ as follows:
    
    \begin{defi}
    	A pencil $\cP\in M(R)$ is \emph{semistable} if the following properties are satisfied:
    	\begin{enumerate}
    		\item The variety $X_K$ over $K$, which is the base locus of $\cP(\Spec K)\in M(K)$, is smooth.
    		\item For every weight system $\rho=(w_1,\ldots, w_n)$, and for every choice of coordinates in $R[x_1,\ldots, x_n]$, the following inequality holds:
    		\begin{align}\label{StabConditionOf(2,2)}
    			\mult_\rho(\cP)\leq\frac{4}{n}\sum_{i=1}^n w_i.
    		\end{align}
    	\end{enumerate}
    \end{defi}

    Let $\bP(V)$ be the projectivization of the space of quadrics $V$ over $R$. Then weight systems naturally act on the set of $K$-points $\bP(V)$. For $f\in\bP(V)(R)$ and a weight system $\rho$, we denote by $\rho\cdot f(\Spec K)$ the corresponding quadric whose coefficients are in $K$. Since $\bP(V)$ is proper, there is an $R$-point $\rho\cdot f$ corresponding to $\rho\cdot f(\Spec K)$.
    
    we define the $t$-valuation $\val_\rho(f)$ of $f$ with repect to a weight system $\rho$ by
    $$\val_\rho(f):=\max\{N\,|\,t^{-N}\cdot(\rho\cdot f(\Spec K))\in R[x_1,\ldots, x_n]\}.$$
    For calculation using semistability condition, we employ the following easy lemma.
    
    \begin{lem}\label{val}
    	Let $\cP\in M(R)$ be a pencil and $\rho$ be a weight system.
    	Take any pair of quadrics $f, g\in R[x_1,\cdots, x_n]$ such that $\{f, g\}$ is an $R$- basis of the pencil $\cP$. Then we have the following inequality:
    	$$\mult_\rho(\cP)\geq\val_\rho(f)+\val_\rho(g).$$
    \end{lem}

      \begin{proof}
      	Take coordinates $\xi_{i, j}$ on the space of single quadrics $V$ which are coefficients of monomials $x_ix_j$ of universal polynomial of degree two. Then, by Pl\"ucker embedding, we obtain coordinates $\xi_{i,j}\wedge\xi_{l,m}$ on $M$. Set
      	$$\rho\cdot f(\Spec K)=\sum_{i,j\in\{1,\cdots n\}}\lambda_{i, j}x_ix_j,\quad\rho\cdot g(\Spec K)=\sum_{i,j\in\{1,\cdots n\}}\mu_{i, j}x_ix_j.$$
      	Then, by definition, every $t$-valuation of $\lambda_{i,j}$ (resp.$\mu_{i,j}$) is more than or equal to $\val_\rho(f)$ (resp. $\val_\rho(g)$). On the other hand, using coordinates defined in the above, $\rho\cdot\cP(\Spec K)\in M(K)$ is written as
      	$$[\cdots:\lambda_{i,j}\mu_{l,m}-\lambda_{l,m}\mu_{i,j}:\cdots].$$
      	Every $t$-valuation of $\lambda_{i,j}\mu_{l,m}-\lambda_{l,m}\mu_{i,j}$ is more than or equal to $\val_\rho(f)+\val_\rho(g)$. Thus we obtain the iniquality by definition of multiplicity of $\cP$.
      \end{proof}
    
    We say that the weight system $\rho=(w_1,\ldots,w_n)$ is \emph{effective} if $w_i\geq0$ for all $i$. In order to check that whether $\cP\in M(R)$ is semistable, we have to only deal effective weight systems. Indeed, if a weight system $\rho=(w_1,\ldots, w_n)$ destabilizes $\cP\in M(R)$, then an effective weight system $\rho'=(w_1-w',\ldots, w_n-w')$ also destabilizes $\cP$, where $w'=\min_{1\leq i\leq n}\{w_i\}$.

  \subsection{Existence of semistable models}

    We give an elementary proof of the existence of semistable $(2, 2)$-complete intersections with a fixed generic fibre $X_K$. The argument is essentially the same as in \cite[4.6]{Kol1}. In this section, we identify quadrics with symmetric matrices and we use letters such as $A, B$ to denote symmetric matrices.
    
    Let $\Phi$ be any homogeneous polynomial of degree $n$ in two variables $\lambda, \mu$. We denote by $D(\Phi(\lambda, \mu))$ the \emph{discriminant} of $\Phi$. In order to show the existence of semistable models of $(2, 2)$-complete intersections, we use the discriminant of $\Delta=\det(\lambda A+\mu B)$, where $A$ and $B$ is a basis of a pencil. Note that, since the discriminant $\Delta$ of pencils depends on the choice of the basis of pencils, so is $D(\det(\lambda A+\mu B))$.
    
    Recall some facts about the discriminant $D(\Phi(\lambda, \mu))$. Let $\Phi(\lambda, \mu)$ be of degree $n$ and set
    $$\Phi(\lambda, \mu)=\sum_{i=0}^n c_i\lambda^i\mu^{n-i}\quad(c_i\in K).$$
    Then the discriminant $D(\Phi(\lambda, \mu))$ is a homogeneous polynomial of degree $2n-2$ in variables $c_i$. On the other hand, giving weight $i$ to each variable $c_i$, $D$ is a quasihomogeneous polynomial of degree $n(n-1)$ in varibles $c_i$.

    \begin{lem}
    	Suppose $A, B\in M_n(K)$. Then we have
    	$$D(\det(\lambda(\xi A)+\mu(\zeta B)))=\xi^{n(n-1)}\zeta^{n(n-1)}D(\det(\lambda A+\mu B))$$
    	for any $\xi, \zeta\in K$.
    \end{lem}
    
      \begin{proof}
      	Let $A=(a_{ij}), B=(b_{ij})$ and set
      	$$\det(\lambda A+\mu B)=\sum_{i=0}^n c_i\lambda^i\mu^{n-i}\quad(c_i\in K).$$
      	Then we can write $c_i$ as a polynomial in variables $a_{ij}, b_{ij}$ $(i,j=1,\ldots, n)$. In each term of $c_m$, the sum of the number of variables $a_{ij}$ (resp. $b_{ij}$) $(i,j=1,\ldots, n)$ is $m$ (resp. $n-m$). In particular, the sum of the number of variables $a_{ij}$ coincides to the weight which we gave above. Thus, if each entry $a_{ij}$ of $A$ is multiplied by $\xi$, then $c_m$ is multiplied by $\xi^m$. Since $D$ is quasihomogeneous of degree $n(n-1)$, we have
      	$$D(\det(\lambda(\xi A)+\mu(B)))=\xi^{n(n-1)}D(\det(\lambda A+\mu B)).$$
      	By a similar calculation, we get the equality.
      \end{proof}

    \begin{lem}\label{DiscIsSemiinv}
        Suppose $A, B\in M_n(K)$ and $P=(p_{ij})\in GL_2(R)$. Then
        $$D(\det(\lambda A+\mu B))=uD(\det(\lambda(p_{11}A+p_{12} B)+\mu(p_{21}A+p_{22}B))$$
        for some $u\in R^\times$.
        In particular, if $\cP\in M(R)$ is a pencil, the $t$-valuation of $D(\Delta)=D((\det(\lambda A+\mu B))$ does not depends on the choice of the basis $A, B$ of the pencil $\cP$.
    \end{lem}
    
      \begin{proof}
    	  This follows from the fact that the discriminant is semiinvariant under the action of $GL_2(K)$ and
    	  \begin{align*}
    	  	& D(\det(\lambda(p_{11}A+p_{12} B)+\mu(p_{21}A+p_{22}B)) \\
    	  	= & D(\det((p_{11}\lambda+p_{21}\mu)A+(p_{12}\lambda+p_{22}\mu)B)).
    	  \end{align*}
      \end{proof}

    Let $A\in M_n(K)$ be a symmetric matrix. We define $\val_t(A)$ to be the minimum of the valuations of all the entries of $A$. This coincides with the minimum of valuations of all the coefficients of the quadric corresponding to the symmetric matrix $A$.

    \begin{lem}\label{LemofDisc}
    	Let $\cP\in M(R)$ be a pencil and $\rho=(w_1,\ldots,w_n)$ be an effective weight system. Take any $R$-basis $\{A, B\}$ (resp. $\{A', B'\}$) of the pencil $\cP$ (resp. $\rho\cdot\cP$). Then, up to a unit in $R$, the following equality holds:
    	$$D(\det(\lambda A'+\mu B'))=t^{n(n-1)(-\mult_\rho(\cP)+(4/n)\sum_{i=1}^n w_n)}D(\det(\lambda A+\mu B)).$$
    \end{lem}
    
      \begin{proof}
      	Firstly we show that
      	\begin{align}\label{1ofLem}
      		D(\det(\lambda A+\mu B))=t^{-4(n-1)\sum_{i=1}^n w_n}D(\det(\lambda F^TAF+\mu F^TBF))
      	\end{align}
      	where $F=E'\diag(t^{w_1},\ldots, t^{w_n})E$ and $E, E'\in GL_n(R)$. ($E$ and $E'$ are any coordinate changes.) Note that $\det F$ is equal to $t^{\sum_{i=1}^n w_n}$ up to a unit of $R$. Using Lemma \ref{DiscIsSemiinv}, the above equality follows from the following easy calculation:
      	\begin{align*}
      		& D(\det(\lambda F^TAF+\mu F^TBF)) \\
      		= & D(\det(F^T(\lambda A+\mu B)F)) \\
      		= & D((\det F)^2\det(\lambda A+\mu B)) \\
      		= & D(\det((\det F)^{2/n}(\lambda A+\mu B))) \\
      		= & (\det F)^{(2/n)\cdot n\cdot(2n-2)}D(\det(\lambda A+\mu B)) \\
      		= & (t^{\sum_{i=1}^n w_n})^{4\cdot(n-1)}D(\det(\lambda A+\mu B)) \\
      		= & t^{4\cdot(n-1)\sum_{i=1}^n w_n}D(\det(\lambda A+\mu B)).
      	\end{align*}
        
        Next, we show that
        \begin{align}\label{2ofLem}
        	D(\det(\lambda F^TAF+\mu F^TBF))=t^{n(n-1)\mult_\rho(\cP)}D(\det(\lambda A'+\mu B'))
        \end{align}
        up to a unit in $R$.
        Note that $\{F^TAF, F^TBF\}$ is a basis of the pencil $\rho\cdot\cP(\Spec K)$.
        
        \noindent
        Set $\widetilde{A}=F^TAF$ and $\widetilde{B}=F^TBF$. Let $N\subset K[x_1,\ldots,x_n]$ be an $R$-module generated by $\widetilde{A}$ and $\widetilde{B}$. If $N$ is equal to the ideal in $R[x_1,\ldots,x_n]$ generated by some basis of pencil $\rho\cdot\cP$ over $R$, the Pl\"ucker coordinates $\widetilde{A}\wedge\widetilde{B}$ and $A\wedge B$ are the same (up to a unit of $R$), thus there is nothing to prove. Suppose not. Then there exists a linear combination (with coefficients in units of $R$) $\widetilde{A}'$ of $\widetilde{A}$ and $\widetilde{B}$ such that $\val_t(\widetilde{A}')\neq0$. Since the weight system $\rho$ is effective, we have $\val_t(\widetilde{A}')>0$. By Lemma \ref{DiscIsSemiinv}, $t$-valuations of $D(\det(\lambda\widetilde{A}+\mu\widetilde{B}))$ and $D(\det(\lambda\widetilde{A}'+\mu\widetilde{B}))$ are the same. Thus, up to a unit in $R$,
        \begin{align*}
        	& D(\det(\lambda\widetilde{A}+\mu\widetilde{B}))=D(\det(\lambda\widetilde{A}'+\mu\widetilde{B})) \\
        	= & t^{n(n-1)\val_t(\widetilde{A}')}D(\det(\lambda(t^{-\val_t(\widetilde{A}')}\widetilde{A}')+\mu\widetilde{B})).
        \end{align*}
        Then $\widetilde{A}'':=t^{-\val_t(\widetilde{A}')}\widetilde{A}'$ satisfies $\val_t\widetilde{A}''=0$. Replacing $\widetilde{A}'$ by $\widetilde{A}''$, we repeat this procedure. The $t$-valuation of the Pl\"ucker coordinte of $\widetilde{A}'\wedge\widetilde{B}'$ decrease in each process. Since the $t$-valuation of the Pl\"ucker coordinate $\widetilde{A}\wedge\widetilde{B}$ is finite, this procedure must stop. Then we get a $R$-basis $\{A'', B''\}$ of the pencil $\rho\cdot\cP$ and
        $$D(\det(\lambda F^TAF+\mu F^TBF))=t^{n(n-1)\mult_\rho(\cP)}D(\det(\lambda A''+\mu B'')).$$
        By Lemma \ref{DiscIsSemiinv}, we may replace $A''$ and $B''$ by any $R$-basis $A'$ and $B'$ of the pencil $\rho\cdot\cP$.
        
        Combining (\ref{1ofLem}) and (\ref{2ofLem}), we get the equality.
      \end{proof}

    \begin{prop}
    	For any smooth $(2, 2)$-complete intersections $X_K$ over $K$, there exists a semistable model $X$ over $R$ whose generic fibre is $X_K$.
    \end{prop}

      \begin{proof}
      	Fix a smooth $(2, 2)$-complete intersection $X_K$ over $K$, and take any model $\cP$ of $(2, 2)$-complete intersection whose generic fibre is $X_K$. If $\cP$ is not semistable, in some coordinate there exists an effective weight system $\rho=(w_1,\ldots, w_n)$ which destabilizes $\cP$. Then by (\ref{StabConditionOf(2,2)}), we have
      	\begin{align}\label{Unstable}
      		\displaystyle\mult_\rho(\cP)>\frac{4}{n}\sum_{i=1}^nw_i.
      	\end{align}
      	Take any $R$-basis $\{A, B\}$ (resp. $\{A', B'\}$) of the pencil $\cP$ (resp. $\rho\cdot\cP$). By Lemma \ref{LemofDisc}, we have
      	$$D(\det(\lambda A'+\mu B'))=t^{n(n-1)(-\mult_\rho(\cP)+(4/n)\sum_{i=1}^n w_n)}D(\det(\lambda A+\mu B)).$$
        By (\ref{Unstable}), $D(\det(\lambda A'+\mu B'))$ is divisible by $tD(\det(\lambda A+\mu B))$ as an element of $R$. If $\rho\cdot\cP$ is not semistable, replace $\cP$ by $\rho\cdot\cP$ and repeat. Then we get an ascending chain of ideals of $R$ such that $(D(\det(\lambda A+\mu B)))\subsetneq(D(\det(\lambda A'+\mu B')))\subseteq\cdots$. Since $R$ is a noetherian ring, this chain must stop and then there is no $K$-valued point of $G$ which destabilizes the output of the above procedure.
      \end{proof}

  \section{Properties of semistable $(2, 2)$-complete intersections}

    Let $\cP\in M(R)$ be a pencil whose corresponding ideal $I\subset R[x_1,\ldots, x_n]$ is generated by distinct two homogeneous degree $2$ polynomials $f, g\in R[x_1,\ldots, x_n]$. Recall that $\cP$ is semistable if and only if:
    \begin{enumerate}
    	\item The generic fibre $X_K$ is a smooth $(2, 2)$-complete intersection over $K$, and
    	\item for every weight system $\rho=(w_1,\ldots, w_n)$, we have that
    	$$\mult_\rho(\cP)\leq\frac{4}{n}\sum_{i=1}^n w_i.$$
    \end{enumerate}

    Set $X:=\{f=g=0\}\subset\bP^{n-1}_R$ and let $\pi:X\ra\Spec R$ be the structure morphism. In this paper, we call $X$ a \emph{semistable model of a $(2, 2)$-complete intersection} if $X$ is defined by an ideal $I$ corresponding to a semistable pencil $\cP\in M(R)$. The \emph{central fibre} $X_0$ is the fibre of $\pi$ at the unique closed point of $\Spec R$.

    \subsection{General properties}

    Since the generic fibre of $\pi$ is integral and $\pi$ is surjective, $\pi$ is a flat morphism. In particular, the central fibre $X_0$ is equidimensional.
    
    \begin{lem}\label{IdealHasNoQuadricWithRank1or2}
    	Let $\cP\in M(R)$ be semistable. Then the ideal $\overline{I}\subset k[x_1,\ldots, x_n]$, which is the image of $I$ under the surjection $R[x_1,\ldots, x_n]\ra k[x_1,\ldots, x_n]$, does not contain rank $1$ or $2$ quadratic forms.
    \end{lem}

      \begin{proof}
      	Assume that $\overline{I}$ contains a rank $1$ (resp. $2$) quadratic form. In some coordinates, it is written as $x_1^2$ (resp. $x_1x_2$). In either case, $\cP$ is destabilized by the weight system $\rho=(1, 0,\ldots, 0)$ by Lemma \ref{val}.
      \end{proof}
    
    \begin{prop}\label{X_0IsIntegral}
    	Let $\cP$ be semistable. Then the central fibre $X_0$ is integral.
    \end{prop}

      \begin{proof}
      	Assume that $X_0$ is not integral. By the degree considerations, there are following two cases:
      	
      	\begin{enumerate}
      		\item $X_0$ contains $(n-3)$-plane.
      		\item $X_0$ is a union of two degree $2$ components $Y_1, Y_2$ (possibly $Y_1=Y_2$ and $X_0$ is non-reduced).
      	\end{enumerate}
      	
      	\noindent
      	Case (1): In some coordinates, $\overline{I}\subset(x_1, x_2)$. By Lemma \ref{val}, $\cP$ is destabilized by $\rho=(1, 1, 0,\ldots, 0)$.
      	
      	\noindent
      	Case (2): If $Y_1\neq Y_2$ (resp. $Y_1=Y_2$), since each degree $2$ component is contained in some $(n-2)$-plane, in some coordinates we may suppose $x_1x_2\in\overline{I}$ (resp. $x_1^2\in\overline{I}$). This contradicts Lemma \ref{IdealHasNoQuadricWithRank1or2}.
      \end{proof}
  
    Combining Proposition \ref{X_0IsIntegral} and smoothness of generic fibre $X_K$, $X$ is regular in codimension $1$. Moreover, since $X$ is a complete intersection, $X$ satisfies the Serre's condition $R_1$ and $S_2$, i.e., $X$ is normal. Using adjunction, $\cO(-K_X)=-\cO(2+2-(n-1)-1)=\cO(n-4)$, thus $-K_X$ is $\pi$-ample.

    \subsection{Singularities of semistable models}

    In this subsection, we investigate the singularities of semistable $(2, 2)$-complete intersections. Firstly, we show that $X$ has isolated singularities. We use the following result which is classification of non-normal integral $(2, 2)$-complete intersentions.
    
    \begin{thm}\cite[Theorem 1.1.]{LPS}\label{ClassificationOf(2,2)}\\
    	Let $k$ be an algebraically closed field with $\chara(k)\neq2$. Suppose that $X\subset\bP^n_k$ is a complete intersection defined by two quadrics $f, g\in k[x_1,\ldots, x_{n+1}]$. If $X$ is irreducible, non-normal, and not a cone, then $X$ can be transformed into one of the following projectively non-equivalent cases:
    	\begin{enumerate}
    		\item $n=3$ and $f=x_1^2+x_2x_3, g=x_4^2+x_1x_3$,
    		\item $n=3$ and $f=x_1^2+x_2x_3, g=x_4^2+x_1x_3+x_1x_4$,
    		\item $n=4$ and $f=x_1^2+x_2x_3, g=x_4^2+x_3x_5$,
    		\item $n=4$ and $f=x_1^2+x_2x_3, g=x_2x_4+x_3x_5$,
    		\item $n=4$ and $f=x_1^2+x_2x_3, g=x_2^2+x_1x_5+x_3x_4$,
    		\item $n=5$ and $f=x_1^2+x_2x_3, g=x_1x_5+x_2x_6+x_3x_4$.
    	\end{enumerate}
    \end{thm}

    Let $\lambda$ be an element of DVR $R$ and $t$ be a uniformizer of $R$. If $\lambda$ is divisible by $t$, we write $\displaystyle\overline{\left(\frac{\lambda}{t}\right)}:=\frac{\lambda}{t}\Big|_{t=0}$.

    \begin{lem}
    	Suppose $n\geq 5$. Let $X\subset\bP^{n-1}_R$ be a semistable model of $(2, 2)$-complete intersections.
    	Then $X$ is regular in codimension $2$.
    	In particular, if $n-1=4$, $X$ has only isolated singularities.
    \end{lem}

      \begin{proof}
      	Choose a coordinate on $\bP^{n-1}_R$ such that the number $m$ of variables to be needed to write $\overline{f}, \overline{g}$ is minimal. Then we may assume that $\overline{f}, \overline{g}\in k[x_1,\ldots, x_m]$. The closed subset $Y_0\subset\bP_k^{m-1}$ which defined by $\overline{f}=\overline{g}=0$ is not a cone. (If $X_0$ is not a cone, then $m=n$ and $X_0=Y_0$.)
      	
      	Assume that $X$ is singular along a codimension $2$ locus. Since the generic fibre $X_K$ is smooth, the singular locus of $X$ is contained in $X_0$, thus $X_0$ is singular along codimension $1$ locus i.e., $X_0$ is non-normal. From the above construction and Proposition \ref{X_0IsIntegral}, we can conclude $Y_0$ is irreducible, non-normal, and not a cone in $\bP^{m-1}_k$. Therefore, in some coordinate the form of the pair $(\overline{f}, \overline{g})$ is one of the cases (1)--(6) in theorem \ref{ClassificationOf(2,2)}. Set
      	$$f=\sum_{i,j\in\{1,\ldots, n\}}\lambda_{i, j}x_ix_j,\quad g=\sum_{i,j\in\{1,\ldots, n\}}\mu_{i, j}x_ix_j.$$
      	We show that each case of Theorem \ref{ClassificationOf(2,2)} cannot occur if $X$ is a semistable $(2,2)$-complete intersection.
      	
      	\textbf{Case (1)}\,: In this case, $X_0$ is singular along $\{x_1=x_3=x_4=0\}$, thus so is $X$ since we assumed that $X$ is singular along codimension $2$ locus.
      	Firstly we calculate the valuations of coefficients of equations on affine chart $x_2=1$. For any $a_5, a_6,\ldots, a_n\in k$, the degree $1$ (with respect to $x_1, x_3, x_4,\ldots  x_n, t$) part of a polynomial $f(x_1, 1, x_3, x_4, x_5+a_5, x_6+a_6,\ldots, x_n+a_n)$ is
      	$$x_3
      	+\overline{\left(\frac{\lambda_{2, 2}}{t}\right)}t
      	+\sum_{j=5}^n a_j\overline{\left(\frac{\lambda_{2, j}}{t}\right)}t
      	+\sum_{i, j\in\{5,\ldots, n\}} a_ia_j\overline{\left(\frac{\lambda_{i, j}}{t}\right)}t$$
      	On the other hand, the degree $1$ (with respect to $x_1, x_3, x_4,\ldots  x_n, t$) part of a polynomial $g(x_1, 1, x_3, x_4, x_5+a_5, x_6+a_6,\ldots, x_n+a_n)$ is
      	$$\overline{\left(\frac{\mu_{2, 2}}{t}\right)}t
      	+\sum_{j=5}^n a_j\overline{\left(\frac{\mu_{2, j}}{t}\right)}t
      	+\sum_{i, j\in\{5,\ldots, n\}} a_ia_j\overline{\left(\frac{\mu_{i, j}}{t}\right)}t.$$
      	Since the point $[0:1:0:0:a_5:\cdots:a_n](t=0)\in X$ is singular, these degree $1$ parts cannot be linearly independent. Therefore
      	$$\overline{\left(\frac{\mu_{2, 2}}{t}\right)}=\overline{\left(\frac{\mu_{2, j}}{t}\right)}=\overline{\left(\frac{\mu_{i, j}}{t}\right)}=0
      	\quad(i, j\in\{5, 6,\ldots, n\}),$$
      	i.e., $t^2\mid\mu_{2, 2} ,\mu_{2, i}, \mu_{i, j}$. 
      	
      	Then, for a weight system $\rho=(1, 0, 1, 1, 0, 0, \ldots, 0)$,
      	$$\mult_\rho(\cP)\geq 1+2=3,$$
      	which is a contradiction.

      	\textbf{Case (2)}\,: In the same way, we see that $t^2\mid\mu_{2, 2} ,\mu_{2, i}, \mu_{i, j}$ for $i, j\in\{5, 6,\ldots, n\}$. So $\rho=(1, 0, 1, 1, 0, 0,\ldots, 0)$ destabilizes $\cP$.

      	\textbf{Case (3)}\,: $X_0$ is singular along $\{x_1=x_3=x_4=0\}$.
      	By a coordinate change
      	$$x_3\mapsto x_3-\lambda_{2, 2}x_2-\sum_{j=5}^n\lambda_{2, j}x_j,$$
      	the coefficients of $x_2^2$ and $x_2x_j$ $(j=5, 6,\ldots, n)$ in $f$ become divisible by $t^2$. Since $\lambda_{2, 2}$ and $\lambda_{2, j}$ $(j=5, 6,\ldots, n)$ are divisible by $t$, the form of $\overline{f}=f|_{t=0}$ and $\overline{g}=g|_{t=0}$ are not changed. Thus we may assume that 
      	$$\overline{\left(\frac{\lambda_{2, 2}}{t}\right)}=\overline{\left(\frac{\lambda_{2, j}}{t}\right)}=0\quad(j=5, 6,\ldots, n).$$
      	For any $a_5, a_6,\ldots, a_n\in k$, the degree $1$ (with respect to $x_1, x_3, x_4,\ldots, x_n, t$) part of a polynomial $f(x_1, 1, x_3, x_4, x_5+a_5, x_6+a_6,\ldots, x_n+a_n)$ is
      	$$x_3+\overline{\left(\frac{\lambda_{2, 2}}{t}\right)}t+\sum_{j=5}^na_j\overline{\left(\frac{\lambda_{2, j}}{t}\right)}t+\sum_{i, j\in\{5,\ldots, n\}}a_ia_j\overline{\left(\frac{\lambda_{i, j}}{t}\right)}t.$$
      	The degree $1$ (with respect to $x_1, x_3, x_4,\ldots, x_n, t$) part of a polynomial $g(x_1, 1, x_3, x_4, x_5+a_5, x_6+a_6,\ldots, x_n+a_n)$ is
      	$$a_5x_3+\overline{\left(\frac{\mu_{2, 2}}{t}\right)}t+\sum_{j=5}^na_j\overline{\left(\frac{\mu_{2, j}}{t}\right)}t+\sum_{i, j\in\{5,\ldots, n\}}a_ia_j\overline{\left(\frac{\mu_{i, j}}{t}\right)}t.$$
      	Since these degree $1$ parts are not linearly independent, for any $a_5, a_6,\ldots, a_n\in k$ we have
      	\begin{align*}
      		& a_5\overline{\left(\frac{\lambda_{2, 2}}{t}\right)}
      		+\sum_{j=5}^na_5a_j\overline{\left(\frac{\lambda_{2, j}}{t}\right)}
      		+\sum_{i, j\in\{5,\ldots, n\}}a_5a_ia_j\overline{\left(\frac{\lambda_{i, j}}{t}\right)} \\
      		& =\overline{\left(\frac{\mu_{2, 2}}{t}\right)}
      		+\sum_{j=5}^na_j\overline{\left(\frac{\mu_{2, j}}{t}\right)}
      		+\sum_{i, j\in\{5,\ldots, n\}}a_ia_j\overline{\left(\frac{\mu_{i, j}}{t}\right)}.
      	\end{align*}
      	This implies
      	\begin{align*}
      		\overline{\left(\frac{\mu_{2, 2}}{t}\right)}=0,\quad
      		\overline{\left(\frac{\mu_{2, j}}{t}\right)}=0\,(j=6,\ldots, n),\\
      		\overline{\left(\frac{\lambda_{2, 2}}{t}\right)}-\overline{\left(\frac{\mu_{2, 5}}{t}\right)}=0,\\
      		\overline{\left(\frac{\mu_{i, j}}{t}\right)}=0\quad(i, j\in\{6,\ldots, n\}),\\
      		\overline{\left(\frac{\lambda_{2, j}}{t}\right)}-\overline{\left(\frac{\mu_{5, j}}{t}\right)}=0,\quad(j=5, 6,\ldots, n),\\
      		\overline{\left(\frac{\lambda_{i, j}}{t}\right)}=0\quad(i, j\in\{5, 6,\ldots, n\}).
      	\end{align*}
      	Since we already have $\overline{\left(\lambda_{2, 2}/t\right)}=\overline{\left(\lambda_{2, j}/t\right)}=0$ for $j=5, 6,\ldots, n$, we obtain
      	$$\overline{\left(\frac{\mu_{2, 5}}{t}\right)}=\overline{\left(\frac{\mu_{5, j}}{t}\right)}=0\quad(j=5, 6,\ldots, n).$$
      	Then, for $\rho=(1, 0, 2, 1, 0, 0,\ldots, 0)$,
      	$$\mult_\rho(\cP)\geq2+2=4.$$
      	Thus $\rho$ destabilizes $\cP$.
      	
      	\textbf{Case (4)}\,: $X_0$ is singular along $\{x_1=x_2=x_3=0\}$. For any $a_5,\ldots, a_n\in k$, the degree $1$ (with respect to $x_1, x_2, x_3, x_5, x_6,\ldots, x_n, t$) part of a polynomial $f(x_1, x_2, x_3, 1, x_5+a_5, x_6+a_6,\ldots, x_n+a_n)$ is
      	$$\overline{\left(\frac{\lambda_{4, 4}}{t}\right)}t
      	+\sum_{j=5}^n a_j\overline{\left(\frac{\lambda_{4, j}}{t}\right)}t
      	+\sum_{i, j\in\{5,\ldots, n\}}a_ia_j\overline{\left(\frac{\lambda_{i, j}}{t}\right)}t.$$
      	The degree $1$ (with respect to $x_1, x_2, x_3, x_5, x_6,\ldots, a_n, t$) part of a polynomial $g(x_1, x_2, x_3, 1, x_5+a_5, x_6+a_6,\ldots, x_n+a_n)$ is
      	$$x_2+a_5x_3
      	+\overline{\left(\frac{\mu_{4, 4}}{t}\right)}t
      	+\sum_{j=5}^n a_j\overline{\left(\frac{\mu_{4, j}}{t}\right)}t
      	+\sum_{i, j\in\{5, 6,\ldots, n\}}a_ia_j\overline{\left(\frac{\mu_{i, j}}{t}\right)}t.$$
      	This implies $t^2\mid\lambda_{i, j}$ for $i, j=4, 5,\ldots, n$.
      	Then, for $\rho=(1, 1, 1, 0, 0,\ldots, 0)$,
      	$$\mult_\rho(\cP)\geq2+1=3,$$
      	so $\rho$ destabilizes $\cP$.
      	
      	\textbf{Case (5), (6)}\,: In each case, $X_0$ is singular along $\{x_1=x_2=x_3=0\}$. By calculation on the affine chart $x_4=1$, one can show that $t^2\mid\lambda_{i, j}$ for $i, j=4, 5,\ldots, n$ as in case (4). Then $\rho=(1, 1, 1, 0, 0,\ldots, 0)$ destabilizes $\cP$.
      \end{proof}

    We use the following lemma to show that semistable threefolds have only terminal singularities.
    
    \begin{lem}\label{SemistableModelHasHypersurfSing}
    	Let $X$ be a semistable model of $(2, 2)$-complete intersections. Then every singular point on $X$ is a hypersurface singularity.
    \end{lem}

      \begin{proof}
      	Let $P=[0:\cdots:0:1]$ be a singular point on $X$. If either $f$ or $g$ is not contained in $(x_1,\ldots, x_{n-1}, t)^2$, $P$ is a hypersurface singularity. Suppose not. Then, for $\rho=(1, 1,\ldots, 1, 0)$,
      	$$\displaystyle\mult_\rho(\cP)\geq 2+2=4>\frac{4}{n}(n-1),$$
      	which is a contradiction.
      \end{proof}

    \section{Singularities on semistable threefolds}
    
    In this section, we prove that semistable del Pezzo fibrations only have terminal singularities, and thus they are standard models of del Pezzo fibrations.
    
    \subsection{Elephants on semistable threefolds}
    
    Let $X$ be a threefold over a field $k$. An \emph{elephant} (resp. \emph{general elephant}) is an element (resp. a general element) of the linear system $|-K_X|$ of the anticanonical divisor. We use the following result.

    \begin{thm}\label{CanonicalElephant}\cite[Corollary 11]{Kol2}\\
    	Let $k$ be an algebraic closed field of arbitrary characteristic. Suppose that $X$ is a scheme over $k$ and $P\in X$ is an isolated normal threefold singularity. If there exists an elephant $E\subset X$ which contains $P$ and $P\in E$ is an isolated du Val singularity, then $P\in X$ is a terminal singularity.
    \end{thm}

    Now we assume that $k$ is an algebraic closed field with $\chara(k)\neq2$. A du Val singularity at $P=(0, 0, 0)$ is an isolated hypersurface singularity in $k[[x, y, z]]$ with multiplicity $2$. Suppose that it is given by an equation $f\in k[[x, y, z]]$. Then it is one of the following (See \cite{Re2}):
    \begin{enumerate}
    	\item If the quadratic part of $f$ is reduced, then $f$ is transformed into the form $xy+z^{n+1}=0$ for some $n\geq1$. In this case, $P$ is called a type $A_n$ singularity. 
    	\item If $f=x^2+g(y, z)$ and the tangent cone of $g(y, z)$ is a cubic with at least two different factors, then $f$ is transformed into the form $x^2+z(y^2+z^{n-2})=0$ for some $n\geq4$. This is called a type $D_n$ singularity.
    	\item Otherwise, it is called a type $E$ singularities. We do not employ them in this paper, so we do not give their description here.
    \end{enumerate}

    By the following lemma, we can take an equation $e\in H^0(X, -K_X)$ which defines an elephant $E$ such that:
    
    $\bullet$ low degree terms of $e$ are in the form of desired one, and
    
    $\bullet$ the generic fibre $E_K$ of $E$ is smooth.

    \begin{lem}\label{SmoothnessOfE_K}
    	Fix $(n_1,\ldots, n_4)\in (\bZ_{\geq 0})^4$ and $(a_1,\ldots, a_4)\in R^4$ such that at least one of $a_i$ is a unit in $R$. Then, for general $(u_1,\ldots, u_5)\in k^5$, the elephant $E$ given by an equation $\sum_{i=1}^{4}(a_i+u_it^{n_i})x_i+u_5tx_5=0$ is smooth in generic fibre.
    \end{lem}

      \begin{proof}
      	Since the generic fibre $X_K$ of $X$ is smooth, the smoothness of the generic fibre of general elephants follows from Bertini's theorem for arbitrary characteristics \cite[12 Corollary]{Kle}. The subfield $k\subset K$ is an infinite field, therefore the set $\{(a_1+u_1t^{n_1},\ldots, a_4+u_1t^{n_4}, u_5)\in K^5|u_1,\ldots, u_5\in k\}$ is dense in $K^5$. Thus the general elephant $E$ given by an equation $\sum_{i=1}^{4}(a_i+u_it^{n_i})x_i+u_5tx_5=0$ is smooth in its generic fibre.
      \end{proof}

    Suppose that $P=[0:0:0:0:1]\in X$ is a singular point. By Lemma \ref{SemistableModelHasHypersurfSing}, we may assume that the equation $g(x_1,\ldots, x_4, 1)$ has nonzero linear terms. Then, replacing $f$ if we need to, we may assume that $f(x_1,\ldots, x_4, 1)$ has no linear terms. In particular, for coefficients $\lambda_{i, j}\in R$ where $f=\sum\lambda_{i, j}x_ix_j$, we have
    \begin{align}
    	t\mid\lambda_{i, 5}\quad(i=1, 2, 3, 4),\qquad t^2\mid\lambda_{5, 5}. \label{CoeffOfx_ix_5}
    \end{align}

    \begin{lem}\label{IsoSingOfElephant}
    	Let $P=[0:0:0:0:1]$ be a singular point on a semistable model $X$, and $E$ be a general elephant given by an equation $x_4+\sum_{i=1}^{4}u_it^{n_i}x_i+u_5tx_5=0$ such that each $n_i\in\bZ$ is large enough.
    	Then $P$ is an isolated singularity as a point of $E$.
    \end{lem}

      \begin{proof}
      	By Lemma \ref{SmoothnessOfE_K}, the generic fibre $E_K$ is smooth, and thus the singular locus of $E$ is contained in its central fibre $E_0$. Since $n_i$ ($1\leq i\leq4$) are large enough, in order to calculate singularities on $E$, we may regard the equation defining the elephant $E$ as $x_4=u tx_5$, where $u\in k$ is general. Set $f_E(x_1, x_2, x_3, x_5):=f(x_1, x_2, x_3, utx_5, x_5)$ and $g_E(x_1, x_2, x_3, x_5):=g(x_1, x_2, x_3, utx_5, x_5)$. By Lemma \ref{SemistableModelHasHypersurfSing}, we may assume that the degree $1$ part of $g_E(x_1, x_2, x_3, 1)$ is not zero. On the other hand, since $P\in E$ is singular, we may assume that the order of $f_E(x_1, x_2, x_3, 1)$ is two.
      	
      	Suppose that $E$ is singular along a curve which contains $P$, and thus so is $E_0$. By the degree considerations, the singular locus of $E_0$ is lines or an irreducible conic.
      	
      	\noindent
      	\textbf{Case 1}: $E_0$ is singular along a line.
      	
      	Note that we can choose $\{f,g\}$ such that (\ref{CoeffOfx_ix_5}) holds. By some coordinate changes, we may assume that the singular locus of $E_0$ is given by $\{x_1=x_2=0\}$ without changing the form of the elephant $x_4=utx_5$ and preserving (\ref{CoeffOfx_ix_5}).
      	The Jacobian matrix of $\{\overline{f_E}, \overline{g_E}\}$ on the line is calculated as follows:
      	
      	$$
      	\begin{pmatrix}
      		\frac{\partial \overline{f_E}}{\partial x_1} & \frac{\partial \overline{f_E}}{\partial x_2} & \frac{\partial \overline{f_E}}{\partial x_3} & \frac{\partial \overline{f_E}}{\partial x_5} \vspace{2mm}\\
      		\frac{\partial \overline{g_E}}{\partial x_1} & \frac{\partial \overline{g_E}}{\partial x_2} & \frac{\partial \overline{g_E}}{\partial x_3} & \frac{\partial \overline{g_E}}{\partial x_5}
      	\end{pmatrix}
        =
        \begin{pmatrix}
        	\overline{\lambda_{1, 3}}x_3 & \overline{\lambda_{2, 3}}x_3 & 0 & 0 \\
        	\overline{\mu_{1, 3}}x_3+\overline{\mu_{1, 5}}x_5 & \overline{\mu_{2, 3}}x_3+\overline{\mu_{2, 5}}x_5 & 0 & 0
        \end{pmatrix}
        $$
        (Note that $\overline{\lambda_{1, 5}}=\overline{\lambda_{2, 5}}=0$ by (\ref{CoeffOfx_ix_5})).
      	
      	\noindent
      	Since the central fibre $E_0$ is singular along the line  $\{x_1=x_2=0\}$, the rank of the Jacobian matrix on the line is $0$ or $1$.
      	
      	\noindent
      	\textbf{Case 1-A}: The rank of the Jacobian matrix is $0$.
      	
      	In this case, $\overline{f_E}$ and $\overline{g_E}$ are contained in $k[x_1, x_2]$. For any $a\in k$, the degree $1$ part (with respect to the variables $x_1, x_2, x_3, t$) of $f_E(x_1, x_2, x_3+a, 1)$ is
      	$$a^2\overline{\left(\frac{\lambda_{3,3}}{t}\right)}t+a\overline{\lambda_{3,4}}ut+a\overline{\left(\frac{\lambda_{3,5}}{t}\right)}t+\overline{\lambda_{4,5}}ut+\overline{\left(\frac{\lambda_{5,5}}{t}\right)}t.$$
      	The degree $1$ part (with respect to the variables $x_1, x_2, x_3, t$) of\\ $g_E(x_1, x_2, x_3+a, 1)$ has the same form:
      	$$a^2\overline{\left(\frac{\mu_{3,3}}{t}\right)}t+a\overline{\mu_{3,4}}ut+a\overline{\left(\frac{\mu_{3,5}}{t}\right)}t+\overline{\mu_{4,5}}ut+\overline{\left(\frac{\mu_{5,5}}{t}\right)}t.$$
      	Since the elephant $E$ is singular along $\{x_1=x_2=0\}$, the above linear parts are not linearly independent. In particular, replacing $f$ with a suitable linear combination of $f$ and $g$, we may assume that the following equalities hold:
      	\begin{align}
      		\overline{\left(\frac{\lambda_{3,3}}{t}\right)}=0,\label{case1-A_lambda33}\\
      		u\overline{\lambda_{3,4}}+\overline{\left(\frac{\lambda_{3,5}}{t}\right)}=0, \label{case1-A_lambda3435} \\
      		u\overline{\lambda_{4,5}}+\overline{\left(\frac{\lambda_{5,5}}{t}\right)}=0. \label{case1-A_lambda4555}
      	\end{align}
        If $\lambda_{3,4}$ is not divisible by $t$, the equality (\ref{case1-A_lambda3435}) does not hold for general $u\in k$ and thus Case 1-A does not occur for general $u\in k$. We can apply the same argument to $\lambda_{4,5}$ and (\ref{case1-A_lambda4555}). If both of $\lambda_{3,4}$ and $\lambda_{4,5}$ are divisible by $t$, we have the following:
        $$t^2\mid\lambda_{3,3},\lambda_{3,5},\lambda_{5,5},\quad
        t\mid\lambda_{3,4},\lambda_{4,5}.$$
        Then, for a weight system $\rho=(1,1,0,1,0)$,
        $$\mult_\rho(\cP)\geq2+1=3,$$
        thus $\rho$ destabilizes $\cP$. Hence Case 1-A does not occur for general $u\in k$.
      	
      	\noindent
      	\textbf{Case 1-B}: The rank of the Jacobian matrix is $1$.
      	
      	The number of nonzero columns of the Jacobian matrix is one or two.
      	
      	\noindent
      	\textbf{Case 1-B-a}: The number of nonzero columns is two.
      	
      	Since the rank of Jacobian matrix is $1$, there exists $[\lambda:\mu]\in\bP_k^1$ such that
      	$$\lambda\cdot\overline{\lambda_{i,3}x_3}+\mu\cdot\left(\overline{\mu_{i,3}}x_3+\overline{\mu_{i,5}}x_5\right)=0\quad(i=1,2).$$
      	Applying some coordinate changes and replacing $f$ and $g$ with suitable linear combinations of $f$ and $g$, the pair $\{\overline{f_E},\overline{g_E}\}$ is written in one of the following two forms:
      	$$
      	\begin{cases}
      		\overline{f_E}=\overline{f_E}^{(2)}(x_1,x_2), \\
      		\overline{g_E}=\overline{g_E}^{(2)}(x_1,x_2)+x_1x_3
      	\end{cases}
        \text{or}\quad
        \begin{cases}
        	\overline{f_E}=\overline{f_E}^{(2)}(x_1,x_2), \\
        	\overline{g_E}=\overline{g_E}^{(2)}(x_1,x_2)+x_1x_5.
        \end{cases}
        $$
        Note that $\{\overline{f_E}, \overline{g_E}\}$ can be transform into the above forms without changing the singular locus $\{x_1=x_2=0\}$.
        For any $a\in k$, consider the degree $1$ part of (with respect to $x_1, x_2, x_3, t$) of $g_E(x_1, x_2, x_3+a, 1)$. In either case, the degree $1$ part has a nonzero term which contains $x_i$ ($i=3$ or $5$). Therefore the degree $1$ part of (with respect to $x_1, x_2, x_3, t$) of $f_E(x_1, x_2, x_3+a, 1)$ must be zero. Then, by the same argument as in Case 1-A, Case 1-B-a does not occur, or (\ref{case1-A_lambda33}), (\ref{case1-A_lambda3435}) and (\ref{case1-A_lambda4555}) hold. In the latter case, a weight system $\rho=(1,1,0,1,0)$ destabilizes $\cP$.

      	\noindent
      	\textbf{Case 1-B-b}: The number of nonzero columns in one.
      	
      	We may assume that $\overline{\lambda_{2,3}}=\overline{\mu_{2,3}}=\overline{\mu_{2,5}}=0$. If $\overline{\lambda_{1,3}}=0$ or $\overline{\mu_{1,5}}=0$, the proof of Case 1-B-a works. Thus we assume that $\overline{\lambda_{1,3}}\neq0$ and $\overline{\mu_{1,5}}\neq0$. Then, applying some coordinate changes and replacing $f$ and $g$ with suitable linear combinations of $f$ and $g$, the pair $\{\overline{f_E},\overline{g_E}\}$ is transformed into the following form:
      	$$
      	\begin{cases}
      		\overline{f_E}=\overline{f_E}^{(2)}(x_1,x_2)+x_1x_3, \\
      		\overline{g_E}=\overline{g_E}^{(2)}(x_1,x_2)+x_1x_5.
      	\end{cases}
        $$
        By the coordinate change $x_1\mapsto x_1-\lambda_{3,3}x_3-\lambda_{3,4}x_4-\lambda_{3,5}x_5$, we may assume
        \begin{align}
        	\overline{\left(\frac{\lambda_{3,3}}{t}\right)}=\overline{\lambda_{3,4}}=\overline{\left(\frac{\lambda_{3,5}}{t}\right)}=0. \label{AssumptionofCase1-B-b}
        \end{align}
        The singular loci of $E$ and $E_0$ are not changed, because $\lambda_{3,3}$ and $\lambda_{3,5}$ are originally divisible by $t$, and the elephant is defined by $x_4=utx_5$.
        For any $a\in k$, the degree $1$ part (with respect to $x_1, x_2, x_3, t$) of $f_E(x_1, x_2, x_3+a, 1)$ is
        $$ax_1+a^2\overline{\left(\frac{\lambda_{3,3}}{t}\right)}t+a\overline{\lambda_{3,4}}ut+a\overline{\left(\frac{\lambda_{3,5}}{t}\right)}t+\overline{\lambda_{4,5}}ut+\overline{\left(\frac{\lambda_{5,5}}{t}\right)}t.$$
        The degree $1$ part (with respect to $x_1, x_2, x_3, t$) of $g_E(x_1, x_2, x_3+a, 1)$ is
        $$x_1+a^2\overline{\left(\frac{\mu_{3,3}}{t}\right)}t+a\overline{\mu_{3,4}}ut+a\overline{\left(\frac{\mu_{3,5}}{t}\right)}t+\overline{\mu_{4,5}}ut+\overline{\left(\frac{\mu_{5,5}}{t}\right)}t.$$
        Since they are not linearly independent, we have
        $$\overline{\left(\frac{\mu_{3,3}}{t}\right)}=0,\quad
        \overline{\left(\frac{\lambda_{3,3}}{t}\right)}=\overline{\mu_{3,4}}u+\overline{\left(\frac{\mu_{3,5}}{t}\right)},$$
        \vspace{-3mm}
        $$\overline{\lambda_{3,4}}u+\overline{\left(\frac{\lambda_{3,5}}{t}\right)}=\overline{\mu_{4,5}}u+\overline{\left(\frac{\mu_{5,5}}{t}\right)},\quad
        \overline{\lambda_{4,5}}u+\overline{\left(\frac{\lambda_{5,5}}{t}\right)}=0.$$
        By (\ref{AssumptionofCase1-B-b}), this implies
        $$\overline{\mu_{3,4}}u+\overline{\left(\frac{\mu_{3,5}}{t}\right)}=0,\quad
        \overline{\lambda_{4,5}}u+\overline{\left(\frac{\lambda_{5,5}}{t}\right)}=0.$$
        Thus, Case 1-B-b does not occur for general $u\in k$, or
        $$t^2\mid\lambda_{3,3},\lambda_{3,5},\lambda_{5,5},\mu_{3,3},\mu_{3,5},\mu_{5,5},\quad 
        t\mid\lambda_{3,4},\lambda_{4,5},\mu_{3,4},\mu_{4,5}.$$
        Then, for $\rho=(2,1,0,1,0)$,
        $$\mult_\rho(\cP)\geq2+2=4,$$
        so $\rho$ destabilizes $\cP$.
        
        Hence each case of Case 1 cannot occur for general $u\in k$.

      	\noindent
      	\textbf{Case 2}: $E_0$ is singular along an irreducible conic.
      	
      	Since $E_0$ has degree $4$, $E_0$ must be an irreducible doubled curve, and thus it is contained in a plane. We may assume that the plane is defined by $x_1=0$. Then we can choose $f$ such that $\overline{f_E}=x_1^2$. By some coodinate changes, $\{\overline{f_E},\overline{g_E}\}$ is written as follows, without changing the form of equation $x_4=utx_5$ defining the elephant $E$:
      	$$
      	\begin{cases}
      		\overline{f_E}(x_1, x_2, x_3, x_5) & =x_1^2, \\
      		\overline{g_E}(x_1, x_2, x_3, x_5)& =\overline{g_E}^{(2)}(x_1,x_2)+x_3x_5+x_1(\overline{\mu_{1, 2}}x_2+\overline{\mu_{1, 3}}x_3+\overline{\mu_{1, 5}}x_5).
      	\end{cases}
      	$$
      	The rank of the quadric $\overline{g_E}^{(2)}$ is $1$ or $2$. Since $E_0$ is irreducible, the coefficient of $x^2$ in $\overline{g_E}^{(2)}$ is nonzero. Hence $\{\overline{f_E},\overline{g_E}\}$ can be transformed into either of the following two forms:
      	\begin{align}\raisetag{-9.5mm}
      		\hspace{10mm}
      		\begin{cases}
      			\overline{f_E}(x_1, x_2, x_3, x_5)=x_1^2, \\
      			\overline{g_E}(x_1, x_2, x_3, x_5)= x_2^2+x_3x_5+x_1(\overline{\mu_{1, 2}}x_2+\overline{\mu_{1, 3}}x_3+\overline{\mu_{1, 5}}x_5)
      		\end{cases} \label{Case2rank1}
      	\end{align}
      	or
      	\begin{align}\raisetag{-9.5mm}
      		\hspace{10mm}
      		\begin{cases}
      			\overline{f_E}(x_1, x_2, x_3, x_5)=x_1^2, \\
      			\overline{g_E}(x_1, x_2, x_3, x_5)= x_2^2+x_3^2+x_3x_5+x_1(\overline{\mu_{1, 2}}x_2+\overline{\mu_{1, 3}}x_3+\overline{\mu_{1, 5}}x_5).
      		\end{cases} \label{Case2rank2}
      	\end{align}
      	If $\displaystyle\overline{\left(\frac{\lambda_{2, 2}}{t}\right)}\neq 0$, replacing $f$ by $\displaystyle f-t\left(\frac{\lambda_{2, 2}}{t}\right)g$ we get $f$ such that $\displaystyle \overline{\left(\frac{\lambda_{2, 2}}{t}\right)}=0$.
      	
      	\vspace{2mm}
      	\noindent
      	\textbf{Case 2-A}: $\{\overline{f_E},\overline{g_E}\}$ is transformed into (\ref{Case2rank1}).
      	
      	For any $a\in k$, the degree $1$ (with respect to $x_1, x_2, x_3, t$) part of a polynomial $f_E(x_1, x_2+a, x_3-a^2, 1)$ is
      	\begin{equation}
      		\begin{split}
      			& \Big{\{}a^2\overline{\left(\frac{\lambda_{2, 2}}{t}\right)}-a^3\overline{\left(\frac{\lambda_{2, 3}}{t}\right)}+a\overline{\lambda_{2, 4}}u+a\overline{\left(\frac{\lambda_{2, 5}}{t}\right)} \\
      			& +a^4\overline{\left(\frac{\lambda_{3, 3}}{t}\right)}-a^2\overline{\lambda_{3, 4}}u-a^2\overline{\left(\frac{\lambda_{3, 5}}{t}\right)}+\overline{\lambda_{4, 5}}u+\overline{\left(\frac{\lambda_{5, 5}}{t}\right)}
      			\Big{\}}t. \label{Case2-A_deg1ofConic}
      		\end{split}
      	\end{equation}
        On the other hand, one can check that the coefficient of $x_2$ in the degree $1$ (with respect to $x_1, x_2, x_3, t$) part of a polynomial $g_E(x_1, x_2+a, x_3-a^2, 1)$ is equal to $2a$.
        Since these degree $1$ parts are not linearly independent, \eqref{Case2-A_deg1ofConic} must be zero. Thus we obtain
        \begin{align*}
        	\overline{\left(\frac{\lambda_{3, 3}}{t}\right)}=0,\quad \overline{\left(\frac{\lambda_{2, 3}}{t}\right)}=0, \\
        	u\overline{\lambda_{3, 4}}+\overline{\left(\frac{\lambda_{3, 5}}{t}\right)}=0,
        	\\
        	u\overline{\lambda_{2, 4}}+\overline{\left(\frac{\lambda_{2, 5}}{t}\right)}=0,
        	\\
        	u\overline{\lambda_{4, 5}}+\overline{\left(\frac{\lambda_{5, 5}}{t}\right)}=0.
        \end{align*}
        Note that we assumed $\overline{(\lambda_{2, 2}/t)}=0$.
        By the same argument as in Case 1, we may assume that $t\mid\lambda_{i, 4}$ and $t^2\mid\lambda_{i, 5}$ for $i=2, 3, 5$.
        Then, for a weight system $\rho=(1, 0, 0, 1, 0)$, we have $$\mult_\rho(\cP)\geq 2+0=2.$$
        Thus $\rho$ destabilizes $\cP$.
        
        \noindent
        \textbf{Case 2-B}: $\{\overline{f_E},\overline{g_E}\}$ is transformed into (\ref{Case2rank2}).
        
        For any $a\in k$, the degree $1$ (with respect to $x_1, x_2, x_3, t$) part of a polynomial $\displaystyle f_E(x_1, x_2+\sqrt{-1}\sqrt{a^2+a}, x_3+a, 1)$ is
        \begin{equation}
        	\begin{split}
        		& \Big{\{}-\left(a^2+a\right)\overline{\left(\frac{\lambda_{2, 2}}{t}\right)}+\sqrt{-1}\cdot a\sqrt{a^2+a}\overline{\left(\frac{\lambda_{2, 3}}{t}\right)} \\
        		& -\sqrt{-1}\sqrt{a^2+a}\,\overline{\lambda_{2, 4}}u-\sqrt{-1}\sqrt{a^2+a}\overline{\left(\frac{\lambda_{2, 5}}{t}\right)} \\
        		& +a^2\overline{\left(\frac{\lambda_{3, 3}}{t}\right)}-a\overline{\lambda_{3, 4}}u-a\overline{\left(\frac{\lambda_{3, 5}}{t}\right)}+\overline{\lambda_{4, 5}}u+\overline{\left(\frac{\lambda_{5, 5}}{t}\right)}
        		\Big{\}}t. \label{Case2-B_deg1ofConic}
        	\end{split}
        \end{equation}
        Since the coefficient of $x_2$ in the degree $1$ (with respect to $x_1, x_2, x_3, t$) part of a polynomial $\displaystyle g_E(x_1, x_2-\sqrt{-1}\sqrt{a^2+a}, x_3+a, 1)$ is equal to $\displaystyle-2\sqrt{-1}\sqrt{a^2+a}$, (\ref{Case2-B_deg1ofConic}) must be zero. Therefore we have
        \begin{align}
        	\overline{\left(\frac{\lambda_{2, 2}}{t}\right)}-\overline{\left(\frac{\lambda_{3, 3}}{t}\right)}=0,\quad\overline{\left(\frac{\lambda_{2, 3}}{t}\right)}=0, \\
        	\overline{\lambda_{2, 4}}u+\overline{\left(\frac{\lambda_{2, 5}}{t}\right)}=0, \\
        	\overline{\left(\frac{\lambda_{2, 2}}{t}\right)}+\overline{\lambda_{3, 4}}u+\overline{\left(\frac{\lambda_{3, 5}}{t}\right)}=0, \\
        	\overline{\lambda_{4, 5}}u+\overline{\left(\frac{\lambda_{5, 5}}{t}\right)}=0.
        \end{align}
        Since we already have $\overline{\left(\lambda_{2, 2}/t\right)}=0$, we may assume that $t\mid\lambda_{i, 4}$ and $t^2\mid\lambda_{i, 5}$ for $i=2, 3, 5$. For $\rho=(1, 0, 0, 1, 0)$, we have
        $$\mult_\rho(\cP)\geq 2+0=2,$$
        so $\rho$ destabilizes $\cP$.
      \end{proof}

    \subsection{Calculation of singularities}
    
    We show that the semistable threefolds have only terminal singularities by concrete calculation.
    
    \begin{prop}
    	Let $X$ be a semistable model of degree $4$ del Pezzo fibrations. Then $X$ has only singularities of type $cA$ or $cD$. In particular, $X$ has only terminal singularities.
    \end{prop}

      \begin{proof}
      	Let $P=[0:0:0:0:1]\in X$ be a singular point. By Theorem \ref{CanonicalElephant} and Lemma \ref{IsoSingOfElephant}, it is sufficient to show that there exists an elephant $E$ which passes through $P$, and $P$ is du Val singularity as a point on $E$. We work on the affine chart $x_5=1$.
      	
      	Let $e\in R[x_1,\ldots, x_5]$ be an equation which defines an elephant $E$ and $e^{(1)}$ be the degree $1$ (with respect to variables $x_1, x_2, x_3, x_4, t)$ part of $e(x_1,\ldots, x_4, 1)$. By Lemma \ref{SmoothnessOfE_K}, we can pick $e$ such that the form of the linear part $e^{(1)}$ is desired one, and generic fibre $E_K$ is smooth. 
      	
      	Let $g^{(1)}$ be the degree $1$ (with respect to variables $x_1, x_2, x_3, x_4, t)$ part of $g(x_1,\ldots, x_4, 1)$. By the Lemma \ref{SemistableModelHasHypersurfSing}, we may assume that $g^{(1)}\neq 0$. Let $\overline{f}^{(2)}$ be the quadratic part of $\overline{f}(x_1,\ldots, x_4, 1)$. By Lemma \ref{IdealHasNoQuadricWithRank1or2}, $\rank\overline{f}^{(2)}$ is $3$ or $4$.

      	Suppose that $\rank \overline{f}^{(2)}=4$.
      	If $g^{(1)}\notin (t)$, in some coordinates, we can write $g^{(1)}=x_4$ or $g^{(1)}=x_4+t$. In this case, we choose $e$ such that $e^{(1)}=g^{(1)}+ut$, where $u\in k$ is general. By Lemma \ref{IsoSingOfElephant}, we may assume that $P\in E$ is an isolated singularity. Then $P$ is a $cA$-type singularity.
      	
      	If $g^{(1)}\in (t)$, we choose $e$ such that $e^{(1)}=x_4+ut$ where $u\in k$ is general. Then $P$ is a $cA$-type singularity.

      	Next, assume that $\rank\overline{f}^{(2)}=3$.
      	Applying some coordinate changes, we may assume that $\overline{f}^{(2)}\in k[x_1, x_2, x_3]$ without changing the singular point $P=[0:0:0:0:1]$.
      	
      	If $g^{(1)}\notin (x_1, x_2, x_3, t)$, in some coordinates, we may write $g^{(1)}=x_4$ or $g^{(1)}=x_4+t$. In this case, we choose $e$ such that $e^{(1)}=g^{(1)}+ut$, where $u\in k$ is general element. Then $P$ is a $cA$-type singularity.
      	
      	If $g^{(1)}\in (t)$, we choose $e$ such that $e^{(1)}=x_4+ut$ where $u\in k$ is general. Then $P$ is a $cA$-type singularity.
      	
      	Assume that $\rank \overline{f}^{(2)}=3$ and $g^{(1)}\in (x_1, x_2, x_3, t)$ but $g^{(1)}\notin (t)$. Applying some coordinate changes, we may assume that $\overline{g}^{(1)}=x_1$.
      	
      	Firstly, consider an elephant $E$ defined by $x_1=u^{-1} tx_5$, where $u\in k$ is general, and the generic fibre $E_K$ is smooth. In the elephant $E$, we have the relations given by $g$ and $x_1=u^{-1} t$. Therefore we can write $x_1$ and $t$ as formal power series with variables $x_2, x_3, x_4$. Set $x_1=\hat{x_1}(x_2, x_3, x_4)$. Note that $\hat{x_1}(x_2, x_3, x_4)$ has order $\geq2$ and, under these relations, $t=u\hat{x_1}(x_2, x_3, x_4)$. On the other hand, in the central fibre, we have another formal power series $x_1=\hat{\overline{x_1}}(x_2, x_3, x_4)\in k[[x_2, x_3, x_4]]$, which is obtained by the relation $\overline{g}=0$.
      	
      	Substituting $x_1=\hat{x_1}$ and $t=u\hat{x_1}$ in the equation $f$, we get a formal power series $\hat{f_E}\in k[[x_2, x_3, x_4]]$, which defines hypersurface singularity at $P=[0:0:0:0:1]$ as a point on $E$.
      	
      	If the rank of the degree $2$ part of $\hat{f_E}$ is more than $2$, then $P$ is $cA$-type singularity. Thus we assume that the rank of degree $2$ part of $\hat{f_E}$ is equal to $1$ (Note that it cannot be zero since $\rank q=3$). In this case, the rank of degree $2$ part of $\overline{f}(\hat{\overline{x_1}}, x_2, x_3, x_4)$ also must be $1$, because the order of each $\hat{x_1}, u\hat{x_1}(=t), \hat{\overline{x_1}}$ is $2$ thus the degree $2$ part of $\hat{f_E}$ and $\overline{f}(\hat{\overline{x_1}}, x_2, x_3, x_4)$ come from only the quadratic part of $f$ which do not contain $x_1$ and $t$. Therefore, by some coordinate changes, $\overline{f}$ is written in the following form:
      	$$\overline{f}=x_1x_2+x_3^2.$$
      	Replacing $g$ by $g-\mu_{3, 3}f$, we may assume that $\mu_{3, 3}=0$.
      	
      	We calculate the terms of $\hat{f_E}$ with degree $\leq 3$ with respect to the variables $x_2, x_3, x_4$. Such terms come from only the following terms in $f$:
      	\begin{align*}
      		x_3^2+\left\{x_1+\overline{\left(\frac{\lambda_{2, 5}}{t}\right)}t\right\}x_2+\overline{\left(\frac{\lambda_{3, 5}}{t}\right)}tx_3+\overline{\left(\frac{\lambda_{4, 5}}{t}\right)}tx_4.
      	\end{align*}
      	Since there is a term $x_3^2$, we may eliminate the term $\overline{\left(\lambda_{3, 5}/t\right)}tx_3$. Then the degree $\leq 3$ part of $\hat{f_E}$ is not changed, because the order of $t=u\hat{x_1}(x_2, x_3, x_4)$ is $2$. Substituting $t=u x_1$, the terms which we have to see become the following:
      	\begin{align}\label{deg2and3partOf_g_when_x_1=xi_t}
      		x_3^2+\left\{1+u\overline{\left(\frac{\lambda_{2, 5}}{t}\right)}\right\}x_1x_2+u\overline{\left(\frac{\lambda_{4, 5}}{t}\right)}x_1x_4.
      	\end{align}
      	
      	To see the degree $\leq 3$ part of $\hat{f_E}$, we substitute $x_1=\hat{x_1}(x_2, x_3, x_4)$ in the above terms. Then we only have to know about the degree $2$ part of $\hat{x_1}(x_2, x_3, x_4)$ which do not have $x_3$. (Note that, by the coordinate change $x_3\mapsto(\text{unit})\cdot x_3$, we can eliminate the terms which contain $x_3^2$.) Such terms come from only the degree $\leq 2$ part of $g$:
      	$$x_1+\overline{\mu_{2, 2}}x_2^2+\overline{\mu_{2, 4}}x_2x_4+\overline{\mu_{4, 4}}x_4^2+\overline{\left(\frac{\mu_{5, 5}}{t}\right)}t.$$
      	Substituting $t=ux_1$, the above terms become
      	$$\left\{1+u\overline{\left(\frac{\mu_{5, 5}}{t}\right)}\right\}x_1+\overline{\mu_{2, 2}}x_2^2+\overline{\mu_{2, 4}}x_2x_4+\overline{\mu_{4, 4}}x_4^2.$$
      	Since $u\in k$ is general, we may assume that $u\overline{\left(\mu_{5, 5}/t\right)}\neq -1$. Thus the degree $2$ part of $\hat{x_1}(x_2, x_3, x_4)$ without $x_3$ is
      	$$u'\left(\overline{\mu_{2, 2}}x_2^2+\overline{\mu_{2, 4}}x_2x_4+\overline{\mu_{4, 4}}x_4^2\right),$$
      	where $u'=-(1+u\overline{\left(\mu_{5, 5}/t\right)})^{-1}$.
      	
      	Assume that $\overline{\mu_{2, 2}}=\overline{\mu_{2, 4}}=\overline{\mu_{4, 4}}=0$. Then, for $\rho=(1, 0, 1, 0, 0)\in G(K)$,
      	$$\mult_\rho(\cP)\geq 1+1=2.$$
      	Thus $\rho$ destabilizes $\cP$. Therefore at least one of $\overline{\mu_{2, 2}}, \overline{\mu_{2, 4}}, \overline{\mu_{4, 4}}$ is nonzero.
      	
      	Substituting $x_1=\hat{x_1}(x_2, x_3, x_4)$ in the degree $\leq 3$ part (\ref{deg2and3partOf_g_when_x_1=xi_t}) of $f$, we get
      	\begin{align*}
      		& & x_3^2+\left\{1+u\overline{\left(\frac{\lambda_{2, 5}}{t}\right)}\right\}\left\{u'\left(\overline{\mu_{2, 2}}x_2^2+\overline{\mu_{2, 4}}x_2x_4+\overline{\mu_{4, 4}}x_4^2\right)\right\}x_2 \\
      		& & +u\overline{\left(\frac{\lambda_{4, 5}}{t}\right)}\left\{u'\left(\overline{\mu_{2, 2}}x_2^2+\overline{\mu_{2, 4}}x_2x_4+\overline{\mu_{4, 4}}x_4^2\right)\right\}x_4 \\
      		& = & x_3^2+\left\{1+u\overline{\left(\frac{\lambda_{2, 5}}{t}\right)}\right\}u'\overline{\mu_{2, 2}}x_2^3 \\
      		& & +\left\{(1+u\overline{\left(\frac{\lambda_{2, 5}}{t}\right)})u'\overline{\mu_{2, 4}}+u\overline{\left(\frac{\lambda_{4, 5}}{t}\right)}u'\overline{\mu_{2, 2}}\right\}x_2^2x_4 \\
      		& & +\left\{(1+u\overline{\left(\frac{\lambda_{2, 5}}{t}\right)})u'\overline{\mu_{4, 4}}+u\overline{\left(\frac{\lambda_{4, 5}}{t}\right)}u'\overline{\mu_{2, 4}}\right\}x_2x_4^2 \\
      		& & +u\overline{\left(\frac{\lambda_{4, 5}}{t}\right)}u'\overline{\mu_{4, 4}}x_4^3.
      	\end{align*}
        If $\displaystyle\overline{\left(\frac{\lambda_{4, 5}}{t}\right)}\neq 0$, for general $u\in k$, the degree $3$ part in the above terms is nonzero and not a cube. thus $P$ is $cD$-type singularity. Hence we may assume that $\displaystyle\overline{\left(\frac{\lambda_{4, 5}}{t}\right)}= 0$. In this case, the degree $\leq 3$ part (\ref{deg2and3partOf_g_when_x_1=xi_t}) of $f$ is
        \begin{align*}
        	x_3^2+\left\{1+u\overline{\left(\frac{\lambda_{2, 5}}{t}\right)}\right\}u'\overline{\mu_{2, 2}}x_2^3 \\
        	+\left\{1+u\overline{\left(\frac{\lambda_{2, 5}}{t}\right)}\right\}u'\overline{\mu_{2, 4}}x_2^2x_4+\left\{1+u\overline{\left(\frac{\lambda_{2, 5}}{t}\right)}\right\}u'\overline{\mu_{4, 4}}x_2x_4^2.
        \end{align*}
        If either $\overline{\mu_{2, 4}}$ or $\overline{\mu_{4, 4}}$ is nonzero, then $P$ is a $cD$-type singularity. Thus we assume that $\overline{\mu_{2, 4}}=\overline{\mu_{4, 4}}=0$. Then $\overline{\mu_{2, 2}}$ is nonzero.
        
        Now take another elephant $E$ defined by $x_2=utx_5$ where $u\in k^\times$ is general. We calculate $\hat{x_1}$ and the low degree part of $f_E$ again. Note that, in this setting, we cannot ignore terms in $f$ and $\hat{x_1}(x_2, x_3, x_4)$ which have the variable $x_3$.
        
        Let $\overline{\lambda_{i, 5}^{(j)}}$ ($i=1,2,\ldots, 5$) be the element of $k$ such that $\lambda_{i, 5}=\sum_{j=0}^\infty\overline{\lambda_{i, 5}^{(j)}}t^j$ in $\hat{R}=k[[t]]$. The degree $\leq 3$ part of $f$ is
        \begin{align*}
        	\overline{\left(\frac{\lambda_{1, 1}}{t}\right)}tx_1^2+x_1x_2+\overline{\left(\frac{\lambda_{1, 3}}{t}\right)}tx_1x_3+\overline{\left(\frac{\lambda_{1, 4}}{t}\right)}tx_1x_4+\overline{\lambda_{1, 5}^{(1)}}tx_1+\overline{\lambda_{1, 5}^{(2)}}t^2x_1 \\
        	+\overline{\left(\frac{\lambda_{2, 2}}{t}\right)}tx_2^2+\overline{\left(\frac{\lambda_{2, 3}}{t}\right)}tx_2x_3+\overline{\left(\frac{\lambda_{2, 4}}{t}\right)}tx_2x_4+\overline{\lambda_{2, 5}^{(1)}}tx_2+\overline{\lambda_{2, 5}^{(2)}}t^2x_2 \\
        	+x_3^2+\overline{\left(\frac{\lambda_{3, 4}}{t}\right)}tx_3x_4+\overline{\lambda_{3, 5}^{(1)}}tx_3+\overline{\lambda_{3, 5}^{(2)}}t^2x_3+\overline{\left(\frac{\lambda_{4, 4}}{t}\right)}tx_4^2+\overline{\lambda_{4, 5}^{(1)}}tx_4+\overline{\lambda_{4, 5}^{(2)}}t^2x_4 \\
        	+\overline{\lambda_{5, 5}^{(2)}}t^2+\overline{\lambda_{5, 5}^{(3)}}t^3.
        \end{align*}
        Substituting $t=ux_2$ in the above, we obtain
        \begin{align*}
        	& \overline{\left(\frac{\lambda_{1, 1}}{t}\right)}(ux_2)x_1^2+x_1x_2+\overline{\left(\frac{\lambda_{1, 3}}{t}\right)}(ux_2)x_1x_3+\overline{\left(\frac{\lambda_{1, 4}}{t}\right)}(ux_2)x_1x_4 \\
        	& +\overline{\lambda_{1, 5}^{(1)}}(ux_2)x_1+\overline{\lambda_{1, 5}^{(2)}}(ux_2)^2x_1 \\
        	& +\overline{\left(\frac{\lambda_{2, 2}}{t}\right)}(ux_2)x_2^2+\overline{\left(\frac{\lambda_{2, 3}}{t}\right)}(ux_2)x_2x_3+\overline{\left(\frac{\lambda_{2, 4}}{t}\right)}(ux_2)x_2x_4 \\
        	& +\overline{\lambda_{2, 5}^{(1)}}(ux_2)x_2+\overline{\lambda_{2, 5}^{(2)}}(ux_2)^2x_2 \\
        	& +x_3^2+\overline{\left(\frac{\lambda_{3, 4}}{t}\right)}(ux_2)x_3x_4+\overline{\lambda_{3, 5}^{(1)}}(ux_2)x_3+\overline{\lambda_{3, 5}^{(2)}}(ux_2)^2x_3 \\
        	& +\overline{\left(\frac{\lambda_{4, 4}}{t}\right)}(ux_2)x_4^2+\overline{\lambda_{4, 5}^{(1)}}(ux_2)x_4+\overline{\lambda_{4, 5}^{(2)}}(ux_2)^2x_4 \\
        	& +\overline{\lambda_{5, 5}^{(2)}}(ux_2)^2+\overline{\lambda_{5, 5}^{(3)}}(ux_2)^3 \\
        	= & \left(1+\overline{\lambda_{1, 5}^{(1)}}u\right)x_1x_2
        	+\left(\overline{\lambda_{2, 5}^{(1)}}+\overline{\lambda_{5, 5}^{(2)}}u\right)ux_2^2
        	+\overline{\lambda_{3, 5}^{(1)}}ux_2x_3+x_3^2+\overline{\lambda_{4, 5}^{(1)}}ux_2x_4 \\
        	& +\overline{\left(\frac{\lambda_{1, 1}}{t}\right)}ux_1^2x_2
        	+\overline{\lambda_{1, 5}^{(2)}}u^2x_1x_2^2
        	+\overline{\left(\frac{\lambda_{1, 3}}{t}\right)}ux_1x_2x_3
        	+\overline{\left(\frac{\lambda_{1, 4}}{t}\right)}ux_1x_2x_4 \\
        	& +\left\{\overline{\left(\frac{\lambda_{2, 2}}{t}\right)}+\overline{\lambda_{2, 5}^{(2)}}u+\overline{\lambda_{5, 5}^{(3)}}u^2\right\}ux_2^3
        	+\left\{\overline{\left(\frac{\lambda_{2, 3}}{t}\right)}+\overline{\lambda_{3, 5}^{(2)}}u\right\}ux_2^2x_3 \\
        	& +\left\{\overline{\left(\frac{\lambda_{2, 4}}{t}\right)}+\overline{\lambda_{4, 5}^{(2)}}u\right\}ux_2^2x_4
        	+\overline{\left(\frac{\lambda_{3, 4}}{t}\right)}ux_2x_3x_4
        	+\overline{\left(\frac{\lambda_{4, 4}}{t}\right)}ux_2x_4^2.
        \end{align*}
    
        Since we assumed that $\overline{\mu_{2, 4}}=\overline{\mu_{4, 4}}=0$, the degree $\leq 2$ part of $g$ is
        \begin{align*}
        	\overline{\mu_{1, 1}}x_1^2+\overline{\mu_{1, 2}}x_1x_2+\overline{\mu_{1, 3}}x_1x_3+\overline{\mu_{1, 4}}x_1x_4+x_1 \\
        	+\overline{\mu_{2, 2}}x_2^2+\overline{\mu_{2, 3}}x_2x_3+\overline{\left(\frac{\mu_{2, 5}}{t}\right)}tx_2+\overline{\mu_{3, 4}}x_3x_4+
        	\overline{\left(\frac{\mu_{3, 5}}{t}\right)}tx_3 \\
        	+\overline{\left(\frac{\mu_{4, 5}}{t}\right)}tx_4+\overline{\mu_{5, 5}^{(1)}}t+\overline{\mu_{5, 5}^{(2)}}t^2.
        \end{align*}
        Therefore the degree $\leq 2$ part of $\hat{x_1}(x_2, x_3, x_4)$ is
        \begin{align*}
        	& -\{
        	\overline{\mu_{1, 1}}(-\overline{\mu_{5, 5}^{(1)}}ux_2)^2+\overline{\mu_{1, 2}}(-\overline{\mu_{5, 5}^{(1)}}ux_2)x_2+\overline{\mu_{1, 3}}(-\overline{\mu_{5, 5}^{(1)}}ux_2)x_3+\overline{\mu_{1, 4}}(-\overline{\mu_{5, 5}^{(1)}}ux_2)x_4 \\
        	& +\overline{\mu_{2, 2}}x_2^2+\overline{\mu_{2, 3}}x_2x_3+\overline{\left(\frac{\mu_{2, 5}}{t}\right)}ux_2^2+\overline{\mu_{3, 4}}x_3x_4+
        	\overline{\left(\frac{\mu_{3, 5}}{t}\right)}ux_2x_3 \\
        	 &+\overline{\left(\frac{\mu_{4, 5}}{t}\right)}ux_2x_4+\overline{\mu_{5, 5}^{(1)}}ux_2+\overline{\mu_{5, 5}^{(2)}}u^2 x_2^2
        	\} \\
        	= & -\overline{\mu_{5, 5}^{(1)}}ux_2
        	+\left\{-\overline{\mu_{1, 1}}\left(\overline{\mu_{5, 5^{(1)}}}\right)^2u^2+\overline{\mu_{1, 2}}\,\overline{\mu_{5, 5^{(1)}}}u-\overline{\mu_{2, 2}}-\overline{\left(\frac{\mu_{2, 5}}{t}\right)}u-\overline{\mu_{5, 5^{(2)}}}u^2\right\}x_2^2 \\
        	& +\left\{\overline{\mu_{1, 3}}\,\overline{\mu_{5, 5^{(1)}}}u-\overline{\mu_{2, 3}}-\overline{\left(\frac{\mu_{3, 5}}{t}\right)}u\right\}x_2x_3
        	+\left\{\overline{\mu_{1, 4}}\,\overline{\mu_{5, 5^{(1)}}}-\overline{\left(\frac{\mu_{4, 5}}{t}\right)}\right\}u x_2x_4
        	-\overline{\mu_{3, 4}}x_3x_4
        \end{align*}
        Substituting this in the degree $\leq 3$ part of $f$, we get the degree $\leq3$ part of $\hat{f_E}(x_2, x_3, x_4)$. It is written as follows:
        \begin{align*}
        	& -\left(1+\overline{\lambda_{1, 5}^{(1)}}u\right)\overline{\mu_{5, 5}^{(1)}}ux_2^2
        	+\left(\overline{\lambda_{2, 5}^{(1)}}+\overline{\lambda_{5, 5}^{(2)}}u\right)ux_2^2
        	+\overline{\lambda_{3, 5}^{(1)}}ux_2x_3
        	+x_3^2+\overline{\lambda_{4, 5}^{(1)}}ux_2x_4 \\
        	& +\left(1+\overline{\lambda_{1, 5}^{(1)}}u\right)\left\{-\overline{\mu_{1, 1}}\left(\overline{\mu_{5, 5^{(1)}}}\right)^2u^2+\overline{\mu_{1, 2}}\,\overline{\mu_{5, 5^{(1)}}}u-\overline{\mu_{2, 2}}-\overline{\left(\frac{\mu_{2, 5}}{t}\right)}u-\overline{\mu_{5, 5^{(2)}}}u^2\right\}x_2^3 \\
        	& +\left(1+\overline{\lambda_{1, 5}^{(1)}}u\right)\left\{\overline{\mu_{1, 3}}\,\overline{\mu_{5, 5^{(1)}}}u-\overline{\mu_{2, 3}}-\overline{\left(\frac{\mu_{3, 5}}{t}\right)}u\right\}x_2^2x_3 \\
        	& +\left(1+\overline{\lambda_{1, 5}^{(1)}}u\right)\left\{\overline{\mu_{1, 4}}\,\overline{\mu_{5, 5^{(1)}}}-\overline{\left(\frac{\mu_{4, 5}}{t}\right)}\right\}ux_2^2x_4
        	-\left(1+\overline{\lambda_{1, 5}^{(1)}}u\right)\overline{\mu_{3, 4}}x_2x_3x_4 \\
        	& +\overline{\left(\frac{\lambda_{1, 1}}{t}\right)}\left(\overline{\mu_{5, 5}^{(1)}}\right)^2u^3x_2^3
        	-\overline{\lambda_{1, 5}^{(2)}}\,\overline{\mu_{5, 5}^{(1)}}u^3x_2^3
        	-\overline{\left(\frac{\lambda_{1, 3}}{t}\right)}\overline{\mu_{5, 5}^{(1)}}u^2x_2^2x_3
        	-\overline{\left(\frac{\lambda_{1, 4}}{t}\right)}\overline{\mu_{5, 5}^{(1)}}u^2x_2^2x_4 \\
        	& +\left\{\overline{\left(\frac{\lambda_{2, 2}}{t}\right)}+\overline{\lambda_{2, 5}^{(2)}}u+\overline{\lambda_{5, 5}^{(3)}}u^2\right\}ux_2^3
        	+\left\{\overline{\left(\frac{\lambda_{2, 3}}{t}\right)}+\overline{\lambda_{3, 5}^{(2)}}u\right\}ux_2^2x_3 \\
        	& +\left\{\overline{\left(\frac{\lambda_{2, 4}}{t}\right)}+\overline{\lambda_{4, 5}^{(2)}}u\right\}ux_2^2x_4
        	+\overline{\left(\frac{\lambda_{3, 4}}{t}\right)}ux_2x_3x_4
        	+\overline{\left(\frac{\lambda_{4, 4}}{t}\right)}ux_2x_4^2.
        \end{align*}
        We have to eliminate the terms which have the variable $x_3$. Namely, they are written as follows:
        \begin{align*}
        	\overline{\lambda_{3, 5}^{(1)}}ux_2x_3+\left(1+\overline{\lambda_{1, 5}^{(1)}}u\right)\left\{\overline{\mu_{1, 3}}\,\overline{\mu_{5, 5^{(1)}}}u-\overline{\mu_{2, 3}}-\overline{\left(\frac{\mu_{3, 5}}{t}\right)}u\right\}x_2^2x_3 \\
        	-\left(1+\overline{\lambda_{1, 5}^{(1)}}u\right)\overline{\mu_{3, 4}}x_2x_3x_4-\overline{\left(\frac{\lambda_{1, 3}}{t}\right)}\overline{\mu_{5, 5}^{(1)}}u^2x_2^2x_3 \\
        	+\left\{\overline{\left(\frac{\lambda_{2, 3}}{t}\right)}+\overline{\lambda_{3, 5}^{(2)}}u\right\}ux_2^2x_3
        	+\overline{\left(\frac{\lambda_{3, 4}}{t}\right)}ux_2x_3x_4.
        \end{align*}
        To eliminate these, we apply the following coordinate change:
        \begin{align*}
        	x_3\mapsto x_3-\frac{1}{2}\Big[\,\overline{\lambda_{3, 5}^{(1)}}ux_2+\left(1+\overline{\lambda_{1, 5}^{(1)}}u\right)\left\{\overline{\mu_{1, 3}}\,\overline{\mu_{5, 5^{(1)}}}u-\overline{\mu_{2, 3}}-\overline{\left(\frac{\mu_{3, 5}}{t}\right)}u\right\}x_2^2 \\
        	-\left(1+\overline{\lambda_{1, 5}^{(1)}}u\right)\overline{\mu_{3, 4}}x_2x_4
        	-\overline{\left(\frac{\lambda_{1, 3}}{t}\right)}\overline{\mu_{5, 5}^{(1)}}u^2x_2^2 \\
        	+\left\{\overline{\left(\frac{\lambda_{2, 3}}{t}\right)}+\overline{\lambda_{3, 5}^{(2)}}u\right\}ux_2^2+\overline{\left(\frac{\lambda_{3, 4}}{t}\right)}ux_2x_4\Big].
        \end{align*}
        By this coordinate change, the coefficients of the monomials $x_2x_4$ and $x_2x_4^2$ are not changed.
        
        If $\overline{\lambda_{4, 5}^{(1)}}\neq0$, the coefficient of $x_2x_4$ is nonzero for general $u\in k$. Thus $P$ is a $cA$-singularity.
        
        Assume that $\overline{\lambda_{4, 5}^{(1)}}=0$, i.e., $t^2\mid\lambda_{4, 5}$. If $\lambda_{4, 4}$ is divisible by $t^2$, for a weight system $\rho=(1, 1, 1, 0, 0)$ we have
        $$\mult_\rho(\cP)\leq2+1=3,$$
        so $\rho$ destabilizes $\cP$. Therefore we have $t^2\nmid\lambda_{4, 4}$ i.e., $\overline{\left(\lambda_{4, 4}/t\right)}$ is nonzero. Then, since all terms of the degree $3$ part in the above is divisible by $x_2$, it cannot be a cube. Thus $P$ is a $cD$-singularity in this case.
      \end{proof}

\end{document}